\documentclass[11pt]{amsart}

\usepackage{amsmath}
\usepackage{amssymb}
\usepackage{graphicx}

\newtheorem{theorem}{Theorem}[section]
\newtheorem{corollary}[theorem]{Corollary}
\newtheorem{lemma}[theorem]{Lemma}
\newtheorem{proposition}[theorem]{Proposition}
\theoremstyle{definition}

\numberwithin{equation}{section}

\title[ Potential Type II ]{A note on potential Type II blowups of axisymmetric solutions to the Navier-Stokes equations}

\title{A note on potential Type II blowups of axisymmetric solutions to the Navier-Stokes equations }

\author[G. Seregin]{Gregory Seregin}
\address{University of Oxford, Mathematical Institute, OxPDE, Oxford, UK, St Petersburg Department of Steklov Mathematical Institute, RAS, Russia, MPI for Mathematics in the Sciences, Leipzig, Germany} 
\email{seregin@maths.ox.ac.uk}



\begin{document}



\rightline{\it Dedicated to Nikolai  Nadirashvili}

\begin{abstract}  In the note, a certain scenario of potential Type II blowups of axisymmetric solutions to the Navier-Stokes equations is considered. The main tool for the treatment of such blowups is the suitable Euler scaling. As the result, the problem of existence of certain Type II blowups is reduced to  Liouville type theorem for ancient axisymmetric solutions to the Euler equations. The class of admissible functions, in which non-trivial solutions are looked for, is defined by a particular scenario of Type II blowup. For example, functions from  such a class must have zero swirl. In the paper, some scenarios of Type II blowup of axisymmetric solutions are excluded by the proving Liouville type theorems  for the Euler equations.
\end{abstract}

\maketitle



 

\setcounter{equation}{0}
\section{Type II Blowups}

The typical object in the study of local regularity of solutions to the Navier-Stokes equations is the so-called suitable weak solution. Let us start with a definition of them, following \cite{CKN}, see also \cite{Lin} and \cite{LS1999} for slightly different versions. Indeed, a pair of functions $v$ and $q$ is called a suitable weak solution  to the classical Navier-Stokes equations, describing the flow of a viscous incompressible fluid in a unique parabolic cylinder $Q$ if $v$ and $q$ have  the following properties:
\begin{itemize} 
\item $v\in L_\infty(-1,0;L_2(B)),\quad \nabla v\in L_2(Q),\quad q\in L_\frac 32(Q);$
	\item 
$\partial_tv+v\cdot\nabla v-\Delta v+\nabla q=0,\quad{\rm div}\,v=0$\\
in $Q$ in the sense of distributions;
\item
for a.a. $t\in ]-1,0[$, the local energy inequality 
$$\int\limits_B\varphi(x,t)|v(x,t)|^2dx+2\int\limits^t_{-1}\int\limits_B\varphi|\nabla v|^2dxd\tau\leq 
$$
$$\leq \int\limits^t_{-1}\int\limits_B (|v|^2(\partial_t\varphi+\Delta\varphi)+v\cdot\nabla\varphi(|v|^2+2q))dxd\tau$$ holds
for all smooth non-negative functions $\varphi$ vanishing in a vicinity of a parabolic boundary of the cylinder $Q$. 
\end{itemize}

Parabolic space-time cylinders and spatial balls are denoted in a standard way:
$Q(r)=B(r)\times ]-r^2,0[$ is a parabolic cylinder 
and $B(r)$ is a spatial ball of radius $r$ centred at the origin $x=0$, $B=B(1)$, and $Q=Q(1)$.

The goal is to find conditions for the existence of a positive number $r$ such that $v\in L_\infty(Q(r))$. If such a number exists, it is said the space-time origin $z=(x,t)=(0,0)=0$ is a regular point of $v$. The further regularity of $v$ in a vicinity of the origin $z=0$ can be obtained with the help of the linear theory. In particular, the velocity field $v$ is going to be H\"older continuous near the origin with the exponent depending on integrability in time of the pressure $q$. 


It is said  that the space-time point $z=0$ is a blowup point of $v$ or a singular point of $v$ if it is not a regular point of $v$. There are two types of blowups: Type I and Type II blowups, see \cite{Seregin2010} for the original definition.
 In particular, the origin  $z=0$ is a Type I blowup of $v$ if it  is a singular point of $v$ and
\begin{equation}
	\label{DefTypeI}
	g=\inf \{\limsup\limits_{r\to0}A(v,r),\limsup\limits_{r\to0}E(v,r),\limsup\limits_{r\to0}C(v,r)\}<\infty.
\end{equation}
In the above definition, the following scale energy quantities are exploited:
$$A(v,r)=\frac1r\sup\limits_{-r^2<t<0}\int\limits_{B(r)}|v(x,t)|^2dx,\quad
E(v,r)=\frac 1r\int\limits_{Q(r)}|\nabla v|^2dz,
$$
$$C(v,r)=\frac 1{r^2}\int\limits_{Q(r)}|\nabla v|^2dz.
$$
The fundamental  property of the those scale energy   quantities is their invariance   with respect to the natural Navier-Stokes scaling, in other words, if $v^\lambda(y,s)=\lambda v(x,t)$, and $q^\lambda(y,s)=\lambda^2 q(x,t)$ with $x=\lambda y$ and $t=\lambda^2s$, then, for example, $A(v^\lambda,r)=A(v,\lambda r)$, etc.

The celebrated Caffarelli-Kohn-Nirenberg theory reads: the origin $z=0$ is a regular point of $v$ if $g$ is sufficiently small, see details in \cite{Seregin2015}. It is, in fact, a particular case of the so-called $\varepsilon$-regularity theory developed for the Navier-Stokes equations. The principal difficulty is, of course, related to non-locality of the Navier-Stokes equations caused by incompressibility condition.  

The important question about Type I blowups is whether the boundedness of $g$ allows blowups or not. In general, it is still an open problem. However, it is known now that axisymmetric suitable weak solutions have no Type I blowups, see \cite{Seregin2020}. So, to complete the regularity theory of axisymmetric solutions, one needs to study potential Type II blowups. The first steps in this direction have been made in the paper \cite{Seregin2023}. Let us recall some definition and statements from there.

We shall say that the origin $z=0$ is a blowup of Type II if it is a singular point and  $g=\infty$. 

Here, in this paper, we shall investigate the following example of  Type II singularity: there exists a sequence  $r_k\to0$ as $k\to\infty$ such that 
\begin{equation}
	\label{growthcond}
M^{3,3}_{2,m_0}(v,r_k):=	\frac 1{r_k^{2m_0}}\int\limits_{Q(r_k)}|v|^3dz\geq c_1>0
\end{equation} for all natural numbers $k$ and some $0< m_0 <1$.

In addition, we restrict ourselves to consideration of the following scenario of the Type II blowup: 
 \begin{equation}
	\label{m-order}
\sup\limits_{0<R\leq 1}	(A_{m_1}(v,R)+D_m(q,R)+E_m(v,R))\leq c<\infty
\end{equation}
with  
$$m=\frac {3m_0}{2+m_0},\qquad m_1=2m-1, $$
 where
$$E_m(v,r)=\frac 1{r^{m}}\int\limits_{Q(r)}|\nabla v|^2dz, \quad A_{m_1}(v,r)=\sup\limits_{-r^2<t<0}\frac 1{r^{m_1}}\int\limits_{B(r)}|v(x,t)|^2dx,$$
$$D_m(q,r)=\frac 1{r^{2m}}\int\limits_{Q(r)}|q|^\frac 32dz.
$$

$$
$$

Obviously, singularities described by \eqref{growthcond} are particular ones. More general Type II blowups are studied in the paper \cite{Seregin2023}  mentioned above.

It is well known now that any suitable weak solution $v$ and $q$ has the following property: 
$\nabla^2v,
\partial_tv\in L_{l_1}(-r^2,0;L_{s_1}(B(r))$ for any $0<r<1$ and for any $s_1, l_1>1$ provided
\begin{equation}
	\label{for second deriv}
	\frac 3{s_1}+\frac 2{l_1}=4.
\end{equation}
Keeping mind that property, it is reasonable to assume additionally boundedness of another scaled quantity, i.e.:
 \begin{equation}
 	\label{secondderiveNsl} 
 	N^{s_1,l_1}(v,r):=\frac 1{r^{\gamma_*l_1}}\int\limits^0_{-r^2}\Big(\int\limits_{B(r)}(|\nabla^2v|^{s_1}+|\partial_tv|^{s_1})dx\Big)^\frac {l_1}{s_1}dt\leq c<\infty
 	\end{equation}
 for all $0<r<1$ and for some $s_1$ and $l_1$ satisfying restriction \eqref{for second deriv} with
 $$\gamma_*(s_1,m)=1-(3/(2s_1)-1)(1-m).
 $$

In the same way as it has been done in the paper \cite{Seregin2023}, one can prove the following statement.
\begin{proposition} 
\label{ancientsolution}	Suppose that a pair $v$ and $q$ is a suitable weak solution to the Navier-Stokes equations in the unit space-time cylinder $Q$. Assume $v$ and $q$ satisfy the conditions \eqref{m-order}, \eqref{growthcond}, and \eqref{secondderiveNsl}. 


Then, there are two functions $u$ and $p$ defined in $Q_-=\mathbb R^3\times]-\infty,0[$, with the following properties:
	$$\sup\limits_{a>0}\Big[\sup\limits_{-a^2<s<0}\frac 1{a^{m_1}}\int\limits_{B(a)}|u(y,s)|^2dy+\frac 1{a^{2m}}\int\limits_{Q(a)}|p|^\frac 32dyds+$$
 \begin{equation}
	\label{basicestimates}
+\frac 1{a^{m}}\int\limits_{Q(a)}|\nabla u|^2dyds+\end{equation}
$$+\frac 1{a^{\gamma_
*l_1}}\int\limits^0_{-a^2}\Big(\int\limits_{B(a)}(|\nabla^2 u|^{s_1}+|\partial_s u|^{s_1})dy\Big)^\frac {l_1}{s_1}ds\Big]\leq c<\infty;$$
\begin{equation}
	\label{Euler}\partial_tu+u\cdot\nabla u+\nabla p=0, \quad{\rm div}\,u=0
\end{equation}in $Q_-=\mathbb R^3\times ]-\infty,0[$ in the sense of distributions;
 
 for a.a. $\tau_0\in ]-\infty,0[$, the local energy inequality
$$\int\limits_{\mathbb R^3}|u(y,\tau_0)|^2\varphi(y,\tau_0)dy\leq$$
\begin{equation}
	\label{enerylocal}
\leq \int\limits_{-\infty}^{\tau_0}\int\limits_{\mathbb R^3}(|u|^2\partial_s\varphi+u\cdot\nabla\varphi(|u|^2+2
p))dyd
\tau
\end{equation}
holds for non-negative $\varphi\in C^\infty_0(\mathbb R^3\times \mathbb R)$;
 
 the function $u$ is not trivial in the following sense
 \begin{equation}
	\label{nontrivial}
	M^{3,3}_{2,m_0}(u,1)\geq c_1/2.	
\end{equation} 
\end{proposition}
The proof of Proposition \ref{ancientsolution}
 is essentially based on the Euler scaling 
$$v(x,t)\longrightarrow\lambda^\alpha v(\lambda x, \lambda^{\alpha+1}t), \quad q(x,t)\longrightarrow\lambda^{2\alpha}q(\lambda x,\lambda^{\alpha+1}t), 
$$ where  $\lambda$ is a positive parameter tending to zero and  
$$\alpha=2-m =\frac{4-m_0}{2+m_0}.
$$

To avoid a certain trivial situation, see paper \cite{Seregin2023}, it is also assumed that
\begin{equation}
	\label{bounds}
	\frac 12\leq m<1\quad\Big (\Longleftrightarrow m_1\geq 0 \Longleftrightarrow\frac 25\leq m_0<1\Big).
\end{equation}

In the rest of the paper, we are going to exploit Proposition \ref{ancientsolution} in order to study the case of suitable weak solutions with axial symmetry.



\setcounter{equation}{0}
\section{Axial Symmetry}
Here and in what follows we are going to work in  cylindrical coordinates $r=|x'|$, with $x'=(x_1,x_2,0)$, $\vartheta$, and $x_3$. So, $v=v_re_r+v_\vartheta e_\theta+v_3e_3$, where $e_r$, $e_\vartheta$, and $e_3$ form the orthonormal  basis of the cylindrical coordinates.

Here, just for convenience, we replace  all spatial balls $B(a)$ with cylinders $\mathcal C(a):=\{|y'|<a,\, |y_3|<a\}$.

In the case of axial symmetry, as it has been shown in previous papers, there is one bounded quantity which is critical, i.e. it is invariant with respect to the natural Navier-Stokes scaling. Here, it is $rv_\vartheta$. Looking at the proof of the Proposition \ref{ancientsolution}, given in the paper \cite{Seregin2023}, where the Euler scaling 
$$v^{\lambda,\alpha)}(y,s)=\lambda^\alpha v(x,t),\quad x=\lambda y,\quad t=\lambda^{\alpha+1}s$$
has been used with parameters $\lambda\to0$ and $\alpha =2-m$, we can observe   the following:  
$$\sqrt{y_1^2+y^2_2}|v_\vartheta^{\lambda,\alpha}(\sqrt{y_1^2+y^2_2},y_3,s)|=\lambda^{\alpha-1}\sqrt{x_1^2+x^2_2}|v_\vartheta(\sqrt{x_1^2+x^2_2},x_3,t)|.$$
The latter means that our limiting Euler equations have no swirl, i.e., $u_\vartheta$   vanishes.

So, the limiting Euler equations must have the form
$$\partial_tu_r+u_ru_{r,r}+u_3u_{r,3}+p_{,r}=0,$$
\begin{equation}
	\label{velocity}
	\partial_tu_3+u_ru_{3,r}+u_3u_{3,3}+p_{,3}=0,
\end{equation}
$$\frac 1r(ru_r)_{,r}+u_{3,3}=0.$$
They also can be reduced to the following different form
$$\partial_tf+u_rf_{,r}+u_3f_{,3}=0,$$
\begin{equation}
	\label{vorticity}
	\Delta \psi-\frac 2r\psi_{,r}=r^2f,
\end{equation}
$$u_r=\frac 1r\psi_{,3},\qquad u_3=-\frac 1r\psi_{,r}.$$
In our case, the only one component of the vorticity 
$$\omega_\vartheta(u)=u_{r,3}-u_{3,r} =rf$$
is not necessary to be zero, i.e., 
$\omega(v)=\omega_\vartheta(v)e_\vartheta$.

Now, our aim is to consider different consequences of the first equation in the system \eqref{vorticity}.
\begin{lemma}
	\label{decaycond}
	Assume that
	\begin{equation}
		\label{decaycond1}
	|v(x,t)|\leq \frac c{|x'|^\alpha}	
	\end{equation}
	for all $z=(x,t)\in Q$ such that $
	|x'|>0$ and let 
	\begin{equation}
		\label{s_1} l_1\leq s_1
	\end{equation}
	 Then
	\begin{equation}
		\label{weakidentity}
		\int\limits_{Q_-}(F(f)\partial_t\varphi+F(f) u\cdot\nabla \varphi)dz=0 
	\end{equation} for all test functions $\varphi(|x'|,x_3,t)\in C^\infty_0(Q_-)$,
	where
	$$F(f)=\Phi(|f|), \quad \Phi(q)= \frac 2{l_1}q^\frac {l_1}2
	$$ and the function $\Phi$ defined for all $q\geq 0$.
\end{lemma}
\begin{proof} Indeed, let 
$$\hat g(x,t)=g(r,x_3,t)= \omega_\theta(v)/r=(v_{r,3}-v_{3,r})/r
$$
and 
$$\hat g^{\lambda,\alpha}(y,s)=g^{\lambda,\alpha}(\rho,y_3,s)= \omega_\theta(v^{\lambda,\alpha})/r=(v^{\lambda,\alpha}_{\rho,3}-v^{\lambda,\alpha}_{3,\rho})/\rho,
$$ where
$$x=\lambda y,\quad r=\lambda \rho,\quad t=\lambda^{\lambda+1}s,$$
$$v^{\lambda,\alpha}(y,s)=\lambda^\alpha v(x,t).
$$
It is known very well that  solution  $v$ is smooth outside the axis of the symmetry $x'=0$ and thus the function $v^{\lambda,\alpha}$ is smooth there ($|y'|>0$) as well. So, denoting $\mathcal C(a,b;h)=\{a<|y'|<b, |y_3|<h\}$ and $Q(a,b,h;T)=\mathcal C(a,b;h)\times ]-T,0[$ and making the change of variables, we arrive at the following identity:
$$\int\limits_{Q(a,b,h;T)}\Big(\hat g^{\lambda,\alpha}\partial_s\varphi+
\hat g^{\lambda,\alpha}\Big(v^{\lambda,\alpha}-\lambda^{\alpha-1}2\frac {y'}{|y'|^2}\Big)\cdot\nabla \varphi
+\lambda^{\alpha-1}1\hat g^{\lambda,\alpha}\Delta\varphi\Big)de=$$
$$
=\lambda^{2\alpha+3}\int\limits_{Q(\lambda a,\lambda b,\lambda h;\lambda^{\alpha+1}T)}\Big(\hat g\partial_s\varphi^{\lambda,\alpha}+
\hat g\Big(v-2\frac {x'}{|x'|^2}\Big)\cdot\nabla \varphi^{\lambda,\alpha}+\hat g\Delta\varphi^{\lambda,\alpha}\Big)dz=0$$
where $\varphi\in C^\infty_0(Q(a,b,h;T))$
and $\varphi^{\lambda,\alpha}(x,t)=\varphi(y,t)$. Following the paper \cite {Seregin2023}, we pass to  the limit in the latter identity as 
$\lambda$ tends to zero. Since $g^{\lambda,\alpha}$  is a component of the tensor $\nabla^2 v^{\lambda, \alpha}$ and  according to  assumption \eqref{secondderiveNsl}, 
$g^{\lambda,\alpha}$  converges to $f$ weakly
in $L_{s_1,l_1}(Q(a,b,h;T))$. Next, we again refer to  paper \cite {Seregin2023}, where  it has been shown that $v^{\lambda,\alpha}\to u$ in $L_3(Q(a,b,h;T))$. In addition, making use of assumption \eqref{decaycond1}, one can prove that $v^{\lambda,\alpha}$ is bounded in $Q(a,b,h;T)$. Indeed,
$$
|v^{\lambda,\alpha}(y,s)|=\lambda^\alpha |v(x,t)|\leq \lambda^\alpha \frac c{|x'|^\alpha}=\frac c{|y'|^\alpha}\leq \frac c{a^\alpha}.$$
Hence, by the last two properties,   $v^{\lambda,\alpha}\to u$ in $L_r(Q(a,b,h;T)$, where $r=\max\{s',l'\}<\infty$. So, after taking the limit, we arrive at the identity
$$\int\limits_{Q(a,b,h;T)}(f\partial_s\varphi+fu\cdot\nabla \varphi)de=0
$$
being valid for all $\varphi \in C^\infty_0(Q(a,b,h;T))$ and for all $0<a<b$ and for all $T>0$. The latter implies the inclusion
$$\partial_sf\in L_{s_1,l_1}(Q(a,b,h;T))$$
and the equation
\begin{equation}
	\label{deriveintime}\partial_sf=-u\cdot\nabla f
\end{equation}
which is true a.e. in $Q(a,b,h;T)$. 

Select a number $\gamma$ so that
\begin{equation}
	\label{choicegamma}
	\frac {m+3}{5-m}<\gamma<1.
\end{equation}

For a fixed positive number $\delta$, let us introduce the function $\Phi_\delta(q)=\frac 1\gamma q^\gamma$ if $q>\delta$ and $\Phi_\delta(q)=\frac 1\gamma q\delta^{\gamma-1}$ if $0<q<\delta$. Then set $F_\delta(f)=\Phi_\delta(|f|)$. It is easy to verify that the function $F_\delta $ is Lipchitz with bounded derivative and the following estimate is valid:
$$F_\delta(f)\leq \frac 1{\gamma}|f|^{\gamma}.$$ Making use of identity \eqref{deriveintime} and the fact that $0<a<r=|y'|<b$, derive  the  following:
$$\partial_sF_\delta(f)=F'_\delta(f)\partial_sf=-F'_\delta(f)u\cdot \nabla f=-u\cdot \nabla F_\delta(f)\in L_{s_1,l_1}(Q(a,b,h;T))
$$
a.e. in $Q(a,b,h;T)$. Taking into account that $0<a<b$, $h$, and $T$ are arbitrary positive numbers, we may pick up a cut-off function $\varphi(x,t)=\varphi_a(|x'|,x_3)\chi(|x'|)\eta(t)$ so that: $\varphi_a(x)=1$ in $\mathcal C(a):=\{x=(x',x_3): \,|x'|<a,\,|x_3|<a\}$ and $\varphi_a(x)=0$ out of  $\mathcal C(2a)$; $\chi(r)=0$ if $0<r<\varepsilon$ and  $\chi(r)=1$ if $2\varepsilon <r$; $\eta\in C^
\infty_0(-T,0)$ with $T<a^2$. In addition it is supposed that  derivatives of the functions $\varphi$ and $\chi$ satisfy the following requirements: $|\nabla \varphi|<c/a$ and $|\chi'|<c/\varepsilon$.

Multiplying the latter identity by the cut-off function $\varphi$ described above and integrating by parts, we find 
\begin{equation}
	\label{eps-identity}\int\limits_{Q_-}(F_\delta(f)\chi\partial_s\psi+F_\delta(f)\chi u\cdot \nabla \psi)dz=-\int\limits_{Q_-}\psi F_\delta(f)u_r\chi'(r)dz=A_\varepsilon
\end{equation}
where $\psi=\varphi_a\eta$. In order to evaluate $A_\varepsilon$, let us change Cartesian coordinates to cylindrical ones. Then, after applying H\"older inequality several times, we have the following chain of estimates:
$$|A_\varepsilon|\leq \frac c\varepsilon\int\limits_{-\infty}^0\int\limits^{2a}_{-2a}\int\limits^{2\varepsilon}_\varepsilon\psi|f|^\gamma|u_r|rdrdx_3dt\leq
$$
$$\leq \frac c\varepsilon\int\limits^0_{-\infty}\eta(t)\Big(\int\limits^{2a}_{-2a}\int\limits^{2\varepsilon}_{\varepsilon}\varphi_a|f|^{s_1}rdrdx_3\Big)^\frac {l_1}{2s_1}\Big(\int\limits^{2a}_{2a}\int\limits^{2\varepsilon}_{\varepsilon}\varphi_a|u_r|^\frac {2s_1}{2s_1-l_1}rdrdx_3\Big)^\frac {2s_1-l_1}{2s_1}dt\leq 
$$
$$\leq \frac c\varepsilon\Big(\int\limits^0_{-\infty}\eta\Big(\int\limits^{2a}_{-2a}\int\limits^{2\varepsilon}_{\varepsilon}\varphi_a|f|^{s_1}rdrdx_3\Big)^\frac {l_1}{s_1}dt\Big)^\frac 12\times$$$$\times\Big(\int\limits^0_{-\infty}\eta\Big(\int\limits^{2a}_{2a}\int\limits^{2\varepsilon}_{\varepsilon}\varphi_a|u_r|^\frac {2s_1}{2s_1-l_1}rdrdx_3\Big)^\frac {2s_1-l_1}{s_1}dt\Big)^\frac 12\leq $$
$$\leq \frac c\varepsilon\Big(\int\limits_{-T}^0\Big(\int\limits_{\mathcal C(\varepsilon,2\varepsilon;2a)}|f|^{s_1}dx\Big)^\frac {l_1}{s_1}dt\Big)^\frac 12\Big(\int\limits_{-T}^0\int\limits_{\mathcal C(\varepsilon,2\varepsilon;2a)}r^2\Big(\frac {|u_r|}r\Big)^2dxdt\Big)^\frac 12\times $$$$\times
|\mathcal C(\varepsilon,2\varepsilon;2a)|
^\frac {s_1-l_1}{2s_1}\leq$$$$\leq c\Big(\int\limits_{-T}^0\Big(\int\limits_{\mathcal C(\varepsilon,2\varepsilon;2a)}|f|^{s_1}dx\Big)^\frac {l_1}{s_1}dt\Big)^\frac 12\Big(\int\limits_{-T}^0\int\limits_{\mathcal C(\varepsilon,2\varepsilon;2a)}\Big(\frac {|u_r|}r\Big)^2dxdt\Big)^\frac 12\times$$$$\times|\mathcal C(\varepsilon,2\varepsilon;2a)|
^\frac {s_1-l_1}{2s_1}.$$
Since $|f|\leq |\nabla^2 u|$ and $|u_r/r|\leq 
|\nabla u|$, we can state that $A_\varepsilon\to0$ as $\varepsilon\to0$ for fixed $\delta$ and $a$. So, passing to the limit in \eqref{eps-identity} and taking into account summability of the integrand there give the variational identity:
\begin{equation}
	\label{eps-identity2}\int\limits_{Q_-}(F_\delta(f)\partial_s\psi+F_\delta(f) u\cdot \nabla \psi)dz=0
\end{equation}
where $\psi=\varphi_a\eta$. For similar reasons, we also can take a limit in \eqref{eps-identity2} as $\delta\to 0$ for fixed $a$. Moreover, it is easy exercise to finish the proof of the lemma.
\end{proof}



Now, let us introduce some abbreviations: 
$$g(t)=\int\limits_{\mathbb R^3}\Phi(|f(x,t)|)dx,\quad g_a(t)=
\int\limits_{\mathbb R^3}\Phi(|f(x,t)|)\phi_a(x)dx.$$
According to Lemma \ref{decaycond},
we have 
$$g_a'(t)=\partial_t\int\limits_{\mathbb R^3}\Phi(|f|)\phi_adx=\int\limits_{\mathbb R^3}\Phi(|f|)u\cdot\nabla\phi_a dx\in L_1(-T_0,0)
$$  for any $T_0>0$.  Fix $T_0>0$ and try to understand under which assumptions on $m$ 
the following is true:
$g'_a\to 0$ in $L_1(-T_0,0)$ as $a\to\infty$. 
In order to evaluate the right hand-side of the latter identity, conditions \eqref{secondderiveNsl} and \eqref{s_1} are going to be used. 
Indeed, repeating arguments from the proof of Lemma \ref{decaycond}, we find
\begin{equation}
	\label{CauchyIneq}
\int\limits_{-T_0}^0|g_a'(t)|dt \leq \frac ca\int\limits^0_{-T_0}\Big(\int\limits_{\mathcal C(2a)}|f|^{s_1}dx\Big)^\frac {l_1}{2s_1}\Big(\int\limits_{\mathcal C(2a)}|u|^\frac {2s_1}{2s_1-l_1}dx\Big)^\frac {2s_1-l_1}{2s_1}dt\leq \end{equation}
$$\leq \frac ca\int\limits^0_{-T_0}\Big(\int\limits_{\mathcal C(2a)}|f|^{s_1}dx\Big)^\frac {l_1}{2s_1}\Big(\int\limits_{\mathcal C(2a)}|u|^2dx\Big)^\frac 12)|\mathcal C(2a)|^\frac {s_1-l_1}{2s_1}\leq $$
$$
\leq \frac caT_0^\frac 12\Big(\int\limits^0_{-T_0}\Big(\int\limits_{\mathcal C(2a)}|f|^{s_1}dx\Big)^\frac {l_1}{s_1}dt\Big)^\frac 12(2a)^\frac {m_1}2A^\frac 12_{m_1}(u,2a)a^{3\frac {s_1-l_1}{2s_1}}$$
Since $|\omega_\vartheta(u)|/r\leq |\nabla^2u|$, the following estimate can be derived from \eqref{CauchyIneq} provided $4a^2>T_0$:
$$\int\limits_{-T_0}^0|g_a'(t)|dt\leq
\frac ca T_0^\frac {1}{2}(2a)^{\frac {m_1}2}A^\frac 12_{m_1}(u,2a)(2a)^{\frac {\gamma_*l_1}2}(N^{s_1,l_1}(u,2a))^\frac {1}2a^{3\frac {s_1-l_1}{2s_1}}\leq
$$
$$\leq  ca^{\frac {m_1}2+\frac {\gamma_*l_1}2+3\frac {s_1-l_1}{2s_1}-1}=ca^{\frac {2m-1}2+\frac {ml_1+6-m-4l_1}2-1}=ca^{\frac {m(1+l_1)+3-4l_1}2}\to0$$
as $a\to \infty$ if it is assumed that 
\begin{equation}
	\label{axirestriction on m}
	m<\frac {4l_1-3}{l_1+1}
	\quad  \Big (\Longleftrightarrow \quad m_0<\frac {2(4l_1-3)}{6-l_1}
	\Big).
\end{equation}

Now, one  can prove easily the following proposition:
\begin{proposition}
	\label{minorpropo}
	Assume that in addition to all conditions of Lemma \ref{decaycond}  restriction  \eqref{axirestriction on m} holds and there exists $t_0\leq 0$ such that 
	\begin{equation}
		\label{assumption}
		g(t_0):=\frac 2{l_1}\int\limits_{\mathbb R^3}\Big(\frac {|\omega_\vartheta(u(x,t_0)|}{r}\Big)^\frac {l_1}2dx<\infty.
	\end{equation}
Then 
\begin{equation}
	\label{preservation}
	g(t):=\frac 2{l_1}\int\limits_{\mathbb R^3}\Big(\frac {|\omega_\vartheta(u(x,t)|}{r}\Big)^\frac {l_1}2dx=g(t_0)
\end{equation}	
for all $t\leq 0$.
\end{proposition}
\begin{proof} Obviously, $g_n(t_0)\to g(t_0)$ as $n\to \infty$. On the other hand,
$$|g_n(t)-g_n(t_0)|\leq \int\limits_{-T_0}^0|g_n'(\tau)|d\tau\to 0$$
as $n\to\infty$. Then by the Beppo Levi theorem, $g_n(t)\to g(t)$ as $n\to\infty$ and moreover \eqref{preservation} holds.
	\end{proof}

It is not very difficult to deduce from  Proposition \ref{minorpropo} the following statement.
\begin{proposition}
	\label{back1} Assume that condition \eqref{axirestriction on m} holds and let in addition
	\begin{equation}
		\label{crtical}
		{\rm ess}\sup\limits_{-1<t<0}\int\limits_{\mathcal C}|v(x,t)|^qdx=:C_0<\infty
	\end{equation} with $q=3/(2-m)\in [2,3[$, see \eqref{bounds}. Then
	 the origin $z=0$ might be a Type II singular point of $v$, described in \eqref{growthcond},  only if at least one of the conditions \eqref{m-order}, \eqref{decaycond1}, and \eqref{s_1} is violated.
	 \end{proposition}
\begin{proof} Using standard arguments for proving partial regularity of suitable weak solutions, see \cite{Seregin2023}, one can establish  the following two facts:
\begin{equation}
\label{everywhere}
\int\limits_{\mathcal C}|v(x,t)|^qdx\leq C_0	\end{equation}	for all $t\in ]-1,0]$;
\begin{equation}
	\label{weakconti}
	\int\limits_
	{\mathcal C(a)}v^{\lambda_k,\alpha}(y,s)\cdot\varphi(y)dy\to \int\limits_
	{\mathcal C(a)}u(y,s)\cdot\varphi(y)dy\end{equation}
in $C([-a^2,0])$ for any $a>0$ and $\varphi\in C^\infty_0(\mathcal C(a))$. Here,
$$v^{\lambda_k,\alpha}(y,s)=\lambda_k^\alpha u(x,t),\quad  x=\lambda_ky,\quad t=\lambda_k^{\alpha+1}s$$
with $\lambda_k\to0$ as $k\to\infty$.
Then, for each fixed test-function $\varphi$, the following estimate can be derived:
$$\Big|\int\limits_
	{\mathcal C(a)}v^{\lambda_k,\alpha}(y,0)\cdot\varphi(y)dy\Big|\leq c(a,q,\varphi)\Big(\int\limits_
	{\mathcal C(a)}|v^{\lambda_k,\alpha}(y,0)|^qdy\Big)^\frac 1q=$$
$$=c(a,q,\varphi)\Big(\int\limits_
	{\mathcal C(\lambda_ka)}|v(x,0)|^qdx\Big)^\frac 1q \to 0$$
as $k\to \infty$. Now, \eqref{everywhere} and \eqref{weakconti} imply
$$\int\limits_
	{\mathcal C(a)}u(y,0)\cdot\varphi(y)dy	=0$$
	for all $\varphi\in C^\infty_0(\mathcal C(a))$ and for all $a>0$. The latter means that $u(\cdot,0)=0$ in $\mathbb R^3$. Hence, $\omega(\cdot,0)=0$ and, by Proposition \ref{minorpropo}, $u$ is an irrotational velocity field, whcih is impossible if $u$ is not trivial, see details in Section 3 of the paper \cite{Seregin2023}.
	\end{proof}
\begin{corollary}
	\label{back2} Assume that condition \eqref{axirestriction on m} holds and let in addition
	\begin{equation}
		\label{crtical2}
		{\rm ess}\sup\limits_{-1<t<0}\int\limits_{\mathcal C}|\nabla \times v(x,t)|^{q_1}dx=:C_0<\infty
	\end{equation} with $q_1=3/(3-m)\in [6/5,3/2[$, see \eqref{bounds}. Then
	 the origin $z=0$ might be a Type II singular point of $v$, described in \eqref{growthcond},  only if if at least one of the conditions \eqref{m-order}, \eqref{decaycond1}, and \eqref{s_1} is violated.
	\end{corollary}


\setcounter{equation}{0}
\section{Self-Similarity and Axial Symmetry}

Here and in what follows, we assume that a potential singular point $z=0$ obeys the additional restriction:
\begin{equation}
	\label{nongrowthcondcyl}
	|v(x,t)|\leq c_2\Big(\frac 1{|x'|^{\alpha_0}(-t)^{(1-\alpha_0)/2}}\Big)^\gamma,\quad \gamma=\frac \alpha{\alpha_0+\frac {1+\alpha}2(1-\alpha_0)},\end{equation}
	for all $x\in \mathcal C$ and for all $-1<t<0$, where $\alpha=2-m$ and a number $0\leq \alpha_0\leq 1$ is given. Since,  inequality
	 \eqref{nongrowthcondcyl} is invariant with respect to the above mentioned Euler scaling, see the beginning of the previous section, the blowup function $u$ must satisfy  a similar estimate:
\begin{equation}
	\label{nongrowthcondcylancient}
	|u(y,s)|\leq c_2\Big(\frac 1{|y'|^{\alpha_0}(-s)^{(1-\alpha_0)/2}}\Big)^\gamma
\end{equation} for all $y=(y',y_3)$ and all $s>-\infty$.

So, let us look for a solution to \eqref{velocity} in a self-similar form
\begin{equation}
	\label{self-similar2}
u(x,t)=\frac 1{(-t)^{\frac \alpha{\alpha+1}}}U(y),\quad p(x,t)=\frac 1{(-t)^{\frac {2\alpha}{\alpha+1}}}P(y),
\end{equation} where
$$y=\frac x{(-t)^{\frac 1{\alpha+1}}}.$$ 
Then, the profile function $U=(U_r,0,U_3)$ obeys the following system
$$\frac \alpha{\alpha+1}U_r+\frac 1{\alpha+1}\Big(rU_{r,r}+y_3U_{r,3}\Big)+U_rU_{r,r}+U_3U_{r,3}+P_{,r}=0,
$$
$$\frac \alpha{\alpha+1}U_3+\frac 1{\alpha+1}\Big(rU_{3,r}+y_3U_{3,3}\Big)+U_rU_{3,r}+U_3U_{3,3}+P_{,3}=0,
$$
$$\frac 1r(rUr)_{,r}+U_{3,3}=0.
$$
With the same change of variables,  the previous system can be transformed  
to the another one:
\begin{equation}
	\label{self-sim-vorticity}F+\Big(\frac
{r}{\alpha+1}+\frac 1r\Psi_{,3}\Big)F_{,r}+\Big(\frac {y_3}{\alpha+1}-\frac 1r\Psi_{,r}\Big)F_{,3}=0,
\end{equation}
$$\Delta \Psi -\frac 1r\Psi_{,r}=r^2F,$$
where 
$$U_r=\frac 1r\Psi_{,3},\qquad U_3=-\frac{1}{r}\Psi_{,r}$$
and 
$$rF=U_{r,3}-U_{3,r}.$$

Let us list all restrictions on $U$:
$$\sup\limits_{b>1}\Bigg(\frac 1{b^{m_1}}\int\limits_{\mathcal C(b)}|U(x)|^2dx+\frac 1{b^m}\int\limits_{\mathcal C(b)}|\nabla U(x)|^2dx+$$
\begin{equation}
\label{profile bounds I}
+\frac 1{b^{\gamma_*l_1}}\Big(\int\limits_{\mathcal C(b)}|\nabla^2U(x)|^{s_1}dx\Big)^\frac {l_1}{s_1}\Bigg)\leq c_0<\infty\end{equation}
where $s_1>1$ must satisfy \eqref{for second deriv} with some $l_1>1$, see arguments in the paper \cite{Seregin2023}, and 
\begin{equation}
\label{profile bounds II}
 |U(y)||y'|^{\gamma\alpha_0}\leq c_2
\end{equation}
for all $y\in \mathbb R^3$. Assuming in addition that $|\nabla v(x,t)|\leq c/(-t)$ for $(x,t)\in \mathcal C\times ]-1,0[$, we notice that it is also invariant with respect to the same Euler scaling. Moreover, after going to self-similar arguments, the following estimate for self-similar profile is valid:
\begin{equation}\label{profile bound III}|\nabla U(y)|\leq c<\infty
\end{equation} for $y\in \mathbb R^3$.
The latter estimate allow us to argue in a way similar to the proof of Lemma \ref{decaycond}.

  
Let $\varphi=\psi(r)\eta(y_3)$ and $\psi$ and $\eta$ be smooth non-negative cut-off functions  such that $\psi(r)=1$ if $0\leq r<R$, $\psi(r)=0$ if $r>2R$, $|\psi'(r)|\leq c/R$ and $\eta(y_3)=1$ if $|y_3|<R$, $\eta(y_3)=0$ if $|y_3|>2R$, $|\eta'(y_3)| \leq c/R$.
 Let $\Phi(F)=\frac 1{s_1}|F|^{s_1}$. Then, we can deduce from \eqref{self-sim-vorticity} that
$$\int\limits_{\mathbb R^3}F\Phi'(F)\varphi dx+\int\limits_{\mathbb R^3}\Big(\frac r{\alpha+1}+\frac 1r\Psi_{,3}\Big)\Phi(F)_{,r}\varphi dx+$$$$+\int\limits_{\mathbb R^3}\Big(\frac {y_3}{\alpha+1}-\frac 1r\Psi_{,r}\Big)\Phi(F)_{,3}\varphi dx=0.$$
Integration by parts gives:
$$\int\limits_{\mathbb R^3}F\Phi'(F)\varphi dx-\frac 3{\alpha+1}\int\limits_{\mathbb R^3}\Phi(F)\varphi dx-\frac 1{\alpha+1}\int\limits_{\mathbb R^3}\Phi(F)(r\varphi_{,r}+y_3\varphi_{,3})dx-$$
$$-\int\limits_{\mathbb R^3}\Phi(F)
U\cdot\nabla\varphi dx=0.
$$
\begin{proposition}
	\label{self-similar result}
	Assume, in addition, that $\alpha_0=0$ and condition \eqref{secondderiveNsl} holds with number $s_1$ satisfying the inequality:
	\begin{equation}
	\label{choice of s1}
	 \frac 3{3-m}=\frac 3{\alpha+1}<s_1 <\frac 32.
\end{equation}
Then $U\equiv0$.
\end{proposition}
\begin{proof}
	Recalling that $\Phi(F)=\frac 1{s_1}|F|^{s_1}$. 
So, we have
$$\Big(s_1-\frac 3{\alpha+1}\Big)\int\limits_{\mathbb R^3}\Phi(F)\varphi dx=$$
$$=\frac 1{\alpha+1}\int\limits_{\mathbb R^3}\Phi(F)(r\varphi_{,r}+y_3\varphi_{,3})dx
+\int\limits_{\mathbb R^3}\Phi(F)
U\cdot\nabla\varphi dx.$$
 
A cut-off function $\varphi$ can be taken so that $r\varphi_{,r}+y_3\varphi_{,3}\leq 0$ and thus, by \eqref{profile bounds II} with $\alpha_0=0$, 

$$I(R):=\int\limits_{\mathcal C(R)}|F|^{s_1} dx\leq 
c(\alpha,s_1)\frac 1R\int\limits_{\mathcal C(2R)\setminus \mathcal C(R)} |F|^{s_1}|U|dx\leq$$
$$\leq c(\alpha,s_1,c_2)\frac 1RI(2R)$$
for all $R>0.$ 
  By bounds $|F|\leq |\nabla \omega(U)|\leq |\nabla^2U|$ and by \eqref{profile bounds I}, we have  the inequality 
$$I(2R)\leq c_0(2R)^{\gamma_*(s_1,m)s_1},$$
which implies 
$$I(R)\leq c_0'R^{\gamma_*s_1-1}
$$ 
for all $R>0$.  Repeating the above arguments, 
we find $I(R)\leq c''_0R^{\gamma_*s_1-2}\to 0$ as $R\to\infty$.
\end{proof}



\section{Discrete Self-Similarity and Axial Symmetry}

Here, we are looking for the function  $u$ of  Proposition \ref{ancientsolution} in the following form
$$u(x,t)=\frac 1{(-t)^{\frac \alpha{\alpha+1}}}U(y,\tau),\quad p(x,t)=\frac 1{(-t)^{\frac {2\alpha}{\alpha+1}}}P(y,\tau),
$$
where 
$$y=\frac x{(-t)^{\frac 1{\alpha+1}}}, \quad \tau =-\ln(-t).
$$
 Then the Euler equations for profile $U$ take the form: 
$$\partial_\tau U_r+\frac \alpha{\alpha+1}U_r+\frac 1{\alpha+1}\Big(rU_{r,r}+y_3U_{r,3}\Big)+
U_rU_{r,r}+U_3U_{r,3}+P_{,r}=0,
$$
$$\partial_\tau U_3+\frac \alpha{\alpha+1}U_3+\frac 1{\alpha+1}\Big(rU_{3,r}+y_3U_{3,3}\Big)+U_rU_{3,r}+U_3U_{3,3}+P_{,3}=0,
$$
$$\frac 1r(rUr)_{,r}+U_{3,3}=0.
$$
One should also add the condition of the periodicity in time for $U$:
$$(U_r,0,U_3)(r,x_3,\tau)=(U_r,0,U_3)(y,\tau+S_0).$$
After introducing new unknown functions, the above system can be transformed as follows:
\begin{equation}\label{discrete self-sim-vorticity}
\partial_\tau	F+\Big(\frac
{r}{\alpha+1}+\frac 1r\Psi_{,3}\Big)F_{,r}+\Big(\frac {y_3}{\alpha+1}-\frac 1r\Psi_{,r}\Big)F_{,3}=0,
\end{equation}
$$\Delta \Psi -\frac 1r\Psi_{,r}=r^2F.$$
The corresponding connections between two systems are given in the following way:
$$U_r=\frac 1r\Psi_{,3},\qquad U_3=-\frac{1}{r}\Psi_{,r}$$
and 
$$rF=U_{r,3}-U_{3,r}.$$

\begin{proposition}
	\label{dicrete-simple}
	 Assume that all conditions of Proposition \ref{minorpropo} 
	  hold. Then $U\equiv0$.
\end{proposition}
\begin{proof}
	By Proposition \ref{minorpropo} and 
	 according to the definition of discrete self-similar solution, we have
$$\int\limits_{\mathbb R^3}\frac{ |{\omega_\vartheta}(u(x,t_0))|^\frac {l_1}2}{|x'|^\frac {l_1}2}dx=\int\limits_{\mathbb R^3}\frac{ |{\omega_\vartheta}(u(x,t))|^\frac {l_1}2}{|x'|^\frac {l_1}2}dx=$$$$=\int\limits_{\mathbb R^3}\frac{ |{\omega_\vartheta}(U(y,\tau))|^\frac {l_1}2}{|y'|^\frac {l_1}2}dy (-t)^\frac {6-l_1 (\alpha+2)}{2(1+\alpha)}.$$
We let $t=t_{(k)}=-\exp\{-\tau_{(k)}\}$, where $\tau_{(k)}=kS_0$ and $k=1,2,..$. By periodicity, 
$U(y, \tau_{(k)})=U(y, \tau_{(0)})$ and therefore
$$\int\limits_{\mathbb R^3}\frac{ |{\omega_\vartheta}(u(x,t_0))|^\frac {l_1}2}{|x'|^\frac {l_1}2}dx=\int\limits_{\mathbb R^3}\frac{ |{\omega_\vartheta}(U(y,\tau_{(0)})|^\frac {l_1}2}{|y'|^\frac {l_1}2}dy (-t_{(k)})^\frac {6-l_1(\alpha+2)}{2(1+\alpha)}\to0.$$
as $k\to\infty$ (since $2>\alpha>1$ and $1<l_1\leq s_1<3/2$). That means $rot\, u=0$, i.e., the flow $u$ is irrotational and thus under our standing assumptions $u$ must be identically equal to zero, see \cite{Seregin2023}. So, $u$ cannot be discrete self-similar under restrictions \eqref{axirestriction on m} and \eqref{assumption}.
\end{proof}
{\bf Acknowledgement} 
 The work is supported by Leverhulme Emeritus Fellowship 2023


\end{document}